\documentclass[12pt, oneside, reqno]{amsart}
\usepackage[top=2.5cm, bottom=2.5cm ,left=3cm, right=3cm]{geometry}

\usepackage{microtype}
\usepackage[
    backend=biber,
    style=alphabetic,
    backref=true,
    url=false,
    doi=false,
    isbn=false
]{biblatex}
\usepackage{mathtools}

\usepackage{iftex}

\usepackage[libertine]{newtxmath}
\iftutex
  \usepackage[no-math]{fontspec}
  \usepackage[nomath]{libertinus}
\else
  \usepackage{libertinus}
\fi

\usepackage{thmtools}
\makeatletter
\@ifundefined{newcounteralias}{}{%
  \renewcommand\thmt@autorefsetup{%
    \@xa\def\csname\thmt@envname autorefname\@xa\endcsname\@xa{\thmt@thmname}}%
}
\makeatother

\usepackage[mathcal]{eucal}
\usepackage{tikz-cd}
\usepackage{enumitem}
\setlist{itemsep=1pt, topsep=4pt}
\usepackage{graphicx}
\usepackage{comment}

\usepackage{xcolor} 
\definecolor{citeblue}{rgb}{0,0.35,0.75}
\usepackage[
    colorlinks=true,
    linkcolor=citeblue,
    citecolor=citeblue,
    urlcolor=citeblue
]{hyperref}
\usepackage[nameinlink,noabbrev]{cleveref}

\numberwithin{equation}{section}

\crefname{equation}{}{}
\Crefname{section}{\S}{}

\theoremstyle{plain}
\newtheorem{theorem}{Theorem}[section]
\newtheorem{lemma}[theorem]{Lemma}
\newtheorem{proposition}[theorem]{Proposition}
\newtheorem{corollary}[theorem]{Corollary}

\theoremstyle{definition}
\newtheorem{definition}[theorem]{Definition}

\newtheorem{remark}[theorem]{Remark}

\DeclareMathOperator{\Mon}{Mon}

\DeclareMathOperator{\divis}{div}
\newcommand{\Z}{\mathbb Z}
\newcommand{\R}{\mathbb R}
\newcommand{\C}{\mathbb C}
\newcommand{\Kah}{\mathcal K}
\newcommand{\Pos}{\mathcal C}
\newcommand{\BK}{\operatorname{BK}}
\newcommand{\FE}{\operatorname{FE}}
\newcommand{\SPE}{\operatorname{SPE}}
\newcommand{\GL}{\operatorname{GL}}
\newcommand{\D}{\operatorname{D}}
\newcommand{\Oth}{\operatorname{O}}

\title{A Strong Global Torelli theorem for OG6 type varieties} 
\author{yun ling}
\date{\today}

\begin{document}

\begin{abstract}
We prove that a hyperk\"ahler manifold of OG6 type is determined, up to isomorphism, by its integral second cohomology endowed with its Hodge structure and Beauville--Bogomolov--Fujiki form. The proof combines the monodromy and chamber theorems of Mongardi--Rapagnetta with the bimeromorphic Torelli theorem. In particular, bimeromorphic OG6 manifolds are biholomorphic.
\end{abstract}
\maketitle

\section{Introduction}

The global Torelli theorem for complex K3 surfaces says that two complex K3 surfaces \(X\) and \(Y\) are isomorphic precisely when there exists an integral Hodge isometry
\[
H^2(X,\mathbb Z)\cong H^2(Y,\mathbb Z)
\]
\cite{PSS,BR}. For higher-dimensional hyperk\"ahler manifolds, the analogous question is considerably more subtle. Birational manifolds need not be isomorphic, even though their second cohomology groups are Hodge isometric. Nor does an integral Hodge isometry, in general, force bimeromorphic equivalence. Namikawa's counterexamples use generalized Kummer fourfolds associated with a complex two-dimensional torus and its dual \cite{Nam}; further counterexamples occur for manifolds of \(K3^{[n]}\)-type when \(n-1\) is not a prime power \cite{MarkmanIntegral}. Thus the second-cohomology Hodge lattice does not, in general, determine either the isomorphism class or the bimeromorphic class.

Verbitsky's global Torelli theorem, together with Huybrechts's work and Markman's Hodge-theoretic formulation, tells us exactly what extra conditions are needed \cite[Theorem~1.3]{Markman}. For two deformation-equivalent irreducible holomorphic symplectic manifolds, bimeromorphicity is equivalent to the existence of a parallel-transport Hodge isometry between their second cohomology groups. When such an isometry sends a K\"ahler class to a K\"ahler class, it is induced by a biholomorphism. Thus the parallel-transport condition and the K\"ahler-cone condition are two separate obstructions to extending the K3 theorem directly.

For OG6 manifolds, Mongardi and Rapagnetta \cite{MR} computed the second-cohomology monodromy group and proved the bimeromorphic global Torelli theorem. In this paper we sharpen their result to a strong global Torelli theorem of K3 type. Our main result is the following

\begin{theorem}\label{thm:main}
Let $X$ and $Y$ be OG6 manifolds. Then $X$ and $Y$ are biholomorphic if and only if there exists an integral Hodge isometry
\[
H^2(X,\mathbb Z)\cong H^2(Y,\mathbb Z).
\]
\end{theorem}

Thus, just as for K3 surfaces, the integral second-cohomology Hodge lattice determines the biholomorphism class of an OG6 manifold. We call this a strong global Torelli theorem, since it classifies such manifolds up to biholomorphism.

The proof hinges on the chamber decomposition of Mongardi--Rapagnetta \cite{MR}. Every wall is orthogonal to a primitive class of square \(-2\), or of square \(-4\) and divisibility \(2\). In either case, reflection in the wall is an integral monodromy operator. After finitely many such reflections, the relevant K\"ahler chambers line up, and the K\"ahler-class criterion in the global Torelli theorem gives the required biholomorphism. The same argument applies to nonprojective OG6 manifolds.

Combining Theorem~\ref{thm:main} with the bimeromorphic global Torelli theorem gives the following consequence.

\begin{corollary}\label{cor:bimeromorphic}
Bimeromorphic OG6 manifolds are biholomorphic. In particular, birational projective OG6 manifolds are isomorphic.
\end{corollary}

This also has an immediate consequence for derived categories. The \(D\)-equivalence conjecture of Bondal--Orlov and Kawamata predicts that birational smooth projective Calabi--Yau varieties have equivalent bounded derived categories of coherent sheaves \cite{BO,Kaw}. For hyperk\"ahler varieties, the conjecture has recently been proved for \(K3^{[n]}\)-type by Maulik--Shen--Yin--Zhang \cite{MSYZ}, and for OG10 type by Hartlieb--Shah \cite{HS}. The following corollary settles the OG6 case directly.

\begin{corollary}\label{cor:derived}
Let \(X\) and \(Y\) be birational projective OG6 manifolds. Then
\[
\D^b(\operatorname{Coh}(X))
\simeq
\D^b(\operatorname{Coh}(Y)).
\]
In particular, the \(D\)-equivalence conjecture holds for projective hyperk\"ahler manifolds of OG6 type.
\end{corollary}

Along with the results for \(K3^{[n]}\)-type and OG10 type, this leaves generalized Kummer type as the only remaining case among the currently known deformation types of irreducible holomorphic symplectic manifolds where the conjecture has not yet been fully established.

\section*{Acknowledgements}

The author is grateful to his advisor, Zhiyuan Li, for suggesting this problem and writing guidance, and to Ruxuan Zhang, Zaiyuan Chen and Ziwei Lu for helpful discussions. The author is supported by the NSFC grants (No.~12425105), the Shanghai Pilot Program for Basic Research (No.~21TQ00) and LMNS.

\medskip
\noindent
\textbf{AI Disclosure.}
During the preparation of this manuscript, the author used AI-based tools as auxiliary aids. In the course of the literature review, GPT-5.6 helped bring Lemma 3.4 to the author’s attention. The author directed the mathematical work, independently verified all arguments and results, and wrote and approved the final version. The author takes full responsibility for the content of the manuscript.

\section{Preliminaries}

	A \emph{hyperk\"ahler manifold} is a simply connected compact K\"ahler manifold $X$ with $H^0(X,\Omega_X^2) = \C \sigma_X$, where $\sigma_X$ is an everywhere non-degenerate holomorphic 2-form on $X$.
	Its second cohomology $H^2(X,\Z)$ is a lattice carrying a canonical integral symmetric bilinear form $q$ of signature $(3,b_2(X)-3)$, called \emph{Beauville--Bogomolov--Fujiki (BBF) form}. It is normalized to be positive on K\"ahler classes \cite{Markman}. We write $q(x)=q(x,x)$ and $\Lambda=H^2(X,\Z)$, equipped with this form.
	
	Let $\Oth^+(\Lambda)$ be the subgroup preserving each of the two connected components of the space of oriented maximal positive definite subspaces of $\Lambda_{\R}$. And let $\Pos_X$ be the component of $\{x\in H^{1,1}(X,\R):q(x)>0\}$ containing the K\"ahler cone $\Kah_X$. For $0\ne\delta\in \Lambda$, by divisibility we mean the positive generator of the subgroup \(q(\delta,\Lambda)\), i.e.,
	\[
	\divis(\delta)\mathbb{Z}=q(\delta,\Lambda).
	\]

    For higher-dimensional hyperk\"ahler manifolds, the global Torelli theorem requires the concept of parallel transport in addition to just integral Hodge isometry.
	\begin{definition}[{\cite[Definition~1.1]{Markman}}]
		Let $X$ and $Y$ be two hyperk\"ahler manifolds.
		\begin{enumerate}
			\item An isomorphism $\varphi: H^k(X,\Z) \to H^k(Y,\Z)$ is a \emph{parallel transport operator} if there is a smooth proper family $\pi: \mathcal X \to B$ of hyperk\"ahler manifolds over an analytic base $B$, points $b_0,b_1 \in B$, isomorphisms $\psi_0:X \xrightarrow{\sim} \mathcal X_{b_0}$ and $\psi_1:Y \xrightarrow{\sim} \mathcal X_{b_1}$, and a continuous path $\gamma:[0,1] \to B$ from $b_0$ to $b_1$, such that parallel transport in the local system $R^k\pi_*\Z$ along $\gamma$ induces $\psi_{1*} \circ \varphi \circ \psi_0^* : H^k(\mathcal X_{b_0},\Z) \xrightarrow{\sim} H^k(\mathcal X_{b_1},\Z)$.
			\item A \emph{monodromy operator} on $H^k(X,\Z)$ is a parallel transport induced by a loop $\gamma$.
			\item The \emph{monodromy group} $\Mon^k(X)$ is the subgroup of $\GL(H^k(X,\Z))$ generated by monodromy operators.
		\end{enumerate}
	\end{definition}

  We have  Markman's global Torelli theorem for hyperk\"ahler manifolds.
    \begin{theorem}[{\cite[Theorem~1.3]{Markman}}]\label{torelli}
        Let $X$ and $Y$ be two deformation equivalent hyperk\"ahler manifolds.
        \begin{enumerate}
            \item $X$ and $Y$ are bimeromorphic, if and only if there exists a parallel transport operator $\varphi: H^2(X,\Z) \xrightarrow{\sim} H^2(Y,\Z)$, which is a integral Hodge isometry.
            \item Let $\varphi$ be the preceding parallel transport integral Hodge isometry, then there exists an isomorphism $f: Y \to X$ such that $\varphi = f^*$, if and only if $\varphi(\Kah_X) \cap \Kah_Y \ne \varnothing$.
        \end{enumerate}
    \end{theorem}

 The BBF form $q$ is deformation invariant, thus $\Mon^2(X) \subset \Oth(\Lambda)$. We denote by $\Mon^2_{\mathrm{Hdg}}(X)$  the subgroup preserving the Hodge structure and let  $\Mon^2_{\mathrm{Bir}}(X)$ be the subgroup induced by self-bimeromorphic maps of $X$. By \cite[\S4-5]{Markman}, they satisfy the following inclusions
	\[
		\Mon^2_{\mathrm{Bir}}(X) \subset \Mon^2_{\mathrm{Hdg}}(X) \subset \Mon^2(X) \subset \Oth^+\bigl(H^2(X,\Z)\bigr).
	\]

    The birational geometry of hyperk\"ahler is governed by some specific wall and chamber decomposition of $\Pos_X$. Such walls are defined as hyperplane
	\[
		\delta^\perp:=\{x \in \Lambda_\R \mid q(x,\delta) = 0\},
	\]
	where $\delta \in \Lambda$ is some specific negative class.
	Recall the birational K\"ahler cone is defined as
	\[
	\BK_X := \bigcup_{f:X\dashrightarrow Y}f^*\Kah_Y,
	\]
	where $f$ ranges over bimeromorphic maps to hyperk\"ahler manifolds. This cone need not be connected.

	\begin{definition}
		\begin{enumerate}
			\item {\cite[Definition~1.2]{Mor}}. Let $\delta$ be a primitive integral $(1,1)$-class.It is a \emph{wall divisor} or \emph{wall class}\footnote{Wall classes are also called \emph{monodromy birationally minimal (MBM)} classes if we adopt the definition of \cite[Definition~1.13]{AV} and ignore the requirement of integral primitivity.} if
			\[
			q(\delta)<0,\qquad
			(u\delta)^\perp\cap\BK_X=\varnothing
			\quad\text{for every }u\in\Mon^2_{\mathrm{Hdg}}(X).
			\]
			\item {\cite[Definitions~5.1 and 6.4]{Markman}}. A \emph{prime exceptional divisor} is a reduced irreducible effective divisor $E$ with $q([E])<0$.
			\item A line bundle $L$ is \emph{stably-prime exceptional (SPE)} if, outside a proper closed analytic subset of the deformation space of the pair, its linear system consists of a prime exceptional divisor.
		\end{enumerate}
	\end{definition}
	
	Both prime exceptional divisors and stably-prime exceptional line bundle are wall divisors, up to some nonzero integral multiples, but the reverse is not true.  Let $\Delta_X$ be the set of primitive wall divisors, $\SPE_X \subset \Delta_X$ be the set of primitive wall divisors that admit a nonzero integral multiple represented by an SPE line bundle. They produce two wall and chamber decompositions, one coarser and one finer.
    
    Before that, let us introduce the \emph{fundamental exceptional chamber}, which is defined as
	\[
	\FE_X=\{x\in\Pos_X \mid q(x,[E])>0
	\text{ for every prime exceptional divisor }E\}.
	\]
    The following result is due to Markman and Amerik-Verbitsky.
    \begin{theorem}\label{thm:general-chambers}
		Let $X$ be a hyperk\"ahler manifold, not necessarily projective.
		\begin{enumerate}
			\item {\cite[Theorem~6.17]{Markman} and \cite[Theorem~3.24]{AV}}. The fundamental exceptional chamber  $\FE_X$ is the connected component containing a K\"ahler class of
			\begin{equation}\label{eq:exceptional-chambers}
				\Pos_X\setminus
				\bigcup_{\delta \in \SPE_X}
				\delta^\perp.
			\end{equation}
			\item {\cite[Theorem~1.19]{AV}} The K\"ahler cone $\Kah_X$ is the connected component containing a K\"ahler class of
			\begin{equation}\label{eq:kahlertype-chambers}
				\Pos_X\setminus\bigcup_{\delta\in\Delta_X}\delta^\perp.
			\end{equation}
		\end{enumerate}
	\end{theorem}

    The connected components of \eqref{eq:exceptional-chambers} are called \emph{exceptional chambers}, and those of \eqref{eq:kahlertype-chambers} are called \emph{K\"ahler-type} or \emph{K\"ahler-Weyl chambers}. Since $\SPE_X\subset\Delta_X$, the latter decomposition refines the former.
	
	Finally, we collect some results from \cite{Markman} and \cite{AV}, to understand the action of $\Mon_{\mathrm{Hdg}}^2(X)$ on these chamber decompositions.
	\begin{lemma}\label{lem:general-chambers}
		Let $X$ be a hyperk\"ahler manifold.
		\begin{enumerate}
			\item The group $\Mon^2_{\mathrm{Hdg}}(X)$ preserves $\Delta_X$ and $\SPE_X$, and hence permutes both chamber decompositions. The exceptional chambers are precisely $u(\FE_X)$, for $u\in\Mon^2_{\mathrm{Hdg}}(X)$.
		
			\item The K\"ahler-type chambers are precisely $u(f^*\Kah_Y)$, where $u\in\Mon^2_{\mathrm{Hdg}}(X)$ and $f:X\dashrightarrow Y$ is bimeromorphic. Those chambers contained in $\FE_X$ are exactly the cones $f^*\Kah_Y$. In particular,
			\[
			      \BK_X\subset\FE_X\subset\overline{\BK_X},
                \qquad
				\BK_X=\FE_X\setminus \bigcup_{\delta\in\Delta_X\setminus\SPE_X}\delta^\perp,
				\qquad 
                \FE_X=\operatorname{Int}\overline{\BK_X}.  
			\]
            Here closure and interior are taken in $H^{1,1}(X,\R)$.
            
			\item The stabilizer of $\FE_X$ is
			\[
				\{u\in\Mon^2_{\mathrm{Hdg}}(X) \mid u(\FE_X)=\FE_X\}
				=\Mon^2_{\mathrm{Bir}}(X).
			\]	
			
		\end{enumerate}
	\end{lemma}
	\begin{proof}
		The invariance in $\Delta_X$ follows from the definition of wall divisors and deformation invariance of SPE classes up to sign \cite[Corollary~6.9]{Markman}, \cite[Theorem~4.10]{AV}. The chamber descriptions in (1) and (2) follow from \cite[Definition~5.10 and Lemma~5.11]{Markman} and \cite[Theorem~6.2]{AV}.
		The remaining part in (2) is by \cite[Proposition~5.6]{Markman}. And statement (3) is \cite[Lemma~5.11(6)]{Markman},
	\end{proof}

\section{Proof of the main theorem}
	From now on, we consider a special hyperkähler, called OG6, originally discovered by O'Grady\cite{O}. It is not necessarily projective.
	Mongardi--Rapagnetta have already obtained all important results of OG6 that we needed.
	\begin{proposition}[{\cite[Theorem~1.4]{MR}}]\label{prop:mr-monodromy}
		For every OG6 manifold $X$, the degree-two monodromy group is maximal:
		\[
		\Mon^2(X)=\Oth^+\bigl(H^2(X,\Z)\bigr).
		\]
	\end{proposition}

	\begin{remark}
		The maximal monodromy phenomenon also occurs in the following known types. For $K3^{[n]}$-type, it holds for $n=1$, i.e., K3 surface, $n=2$ or $n-1$ is a prime power \cite[Lemma~9.2]{Markman}. It also holds for OG10 by \cite[Theorem~5.4]{On}. For generalized Kummer type of dimension $2n\ge4$, it always fails, see \cite[Theorem~1.4]{MaK}. By the way, this is also exactly the list for which the classical global bimeromorphic Torelli holds, since the maximal monodromy phenomenon allows us to remove the parallel transport condition in Theorem~\ref{torelli}.
	\end{remark}

    The wall divisors on OG6 are surprisingly simple. The following proposition gives a full classification.
    \begin{proposition}[{\cite[Proposition~6.8]{MR}}]\label{prop:mr-walls}
		Let $X$ be an OG6 manifold and let $\delta\in H^{1,1}(X,\Z)$ be primitive, then $\delta$ is a wall divisor if and only if
		\[
			(q(\delta),\divis(\delta)) \in\{(-4,2),(-2,2),(-2,1)\}.
		\]
		In particular, some nonzero integral multiple of $\delta$ is the class of a stably prime-exceptional divisor if and only if its numerical type is $(-4,2)$ or $(-2,2)$.
	\end{proposition}

    The key observation is: these two results imply that all wall reflections of an OG6 manifold, although purely arithmetic in construction, can be lifted to Hodge monodromy operators.
	\begin{lemma}\label{lem:reflection}
		For every $\delta\in\Delta_X$, there is a reflection with respect to the wall $\delta^\perp$,
		\[
		s_\delta(x)=x-\frac{2q(x,\delta)}{q(\delta)}\delta.
		\]
		\begin{enumerate}
			\item $s_\delta \in \Mon^2_{\mathrm{Hdg}}(X)$, and it preserves $\Delta_X$, and their numerical types.
			\item If two chambers are adjacent along $\delta^\perp$, then $s_\delta$ exchanges them.
			\item $s_\delta\in\Mon^2_{\mathrm{Bir}}(X)$ if and only if $\delta^\perp\cap\FE_X\ne\varnothing$.
		\end{enumerate}
	\end{lemma}
	
	\begin{proof}
		The reflection preserves $q$, and its formula gives
		\[
		s_\delta\text{ is integral}
		\quad\Longleftrightarrow\quad
		q(\delta)\mid 2\divis(\delta).
		\]
		All three numerical types in Proposition~\ref{prop:mr-walls} satisfy this condition.
		
		Choose $0\ne\sigma\in H^{2,0}(X)$ and $p\in\delta^\perp\cap\Pos_X$. As $\delta$ is of type $(1,1)$, the reflection fixes $H^{2,0}(X) = \mathbb{C}\sigma$ and is a Hodge isometry. Further, the reflection fixes the positive three-plane $\langle\Re\sigma,\Im\sigma,p\rangle_{\R}$, so it belongs to $\Oth^+(\Lambda)$.
		Therefore, Proposition~\ref{prop:mr-monodromy}  gives $s_\delta \in \Mon^2_{\mathrm{Hdg}}(X)$.
		
		An integral Hodge isometry preserves primitivity, square and
		divisibility.Consequently $s_\delta$ permutes $\Delta_X$ and wall arrangement, and hence proves (1).
		
		For (2), choose a point in the relative interior of the common wall that lies on no other wall. By local finiteness, a sufficiently small $s_\delta$-invariant neighbourhood meets no other wall.
		The reflection exchanges its two sides. Since it also permutes the chambers, it exchanges the two adjacent chambers.
		
		For (3), suppose that $p\in\delta^\perp\cap\FE_X$. Since $s_\delta(p)=p$, the exceptional chambers $\FE_X$ and $s_\delta(\FE_X)$ intersect. By Lemma~\ref{lem:general-chambers}, $s_\delta(\FE_X)=\FE_X$ and then $s_\delta\in\Mon^2_{\mathrm{Bir}}(X)$.
		Conversely, such a reflection preserves $\FE_X$. For any $x\in\FE_X$, convexity gives $(x+s_\delta (x))/2\in\FE_X\cap\delta^\perp$.
	\end{proof}

	\begin{remark}
	By \cite[Proposition~6.2]{Markman}, it is known that wall reflections given by $\SPE_X$ are automatically integral, we even know that $s_\delta \in \Mon^2_{\mathrm{Hdg}}(X)\setminus\Mon^2_{\mathrm{Bir}}(X)$, if $\delta \in \SPE_X$. However, not all wall divisors on general hyperk\"ahler manifolds could give integral reflections.
	For example, the four numerical types of wall divisors on OG10 are as follows \cite[Propositions~3.1 and 5.4]{MO}:
    \[
    \begin{array}{c|c|c}
        (q(\delta),\divis(\delta))
        &\text{SPE multiple}
        & s_\delta\text{ integral}\\[3pt]\hline
        (-2,1)&\text{yes}&\text{yes}\\
        (-6,3)&\text{yes}&\text{yes}\\
        (-4,1)&\text{no}&\text{no}\\
        (-24,3)&\text{no}&\text{no}
    \end{array}
    \]
	\end{remark}

	Now we begin the proof of the main theorem.
	\begin{proof}[Proof of Theorem~\ref{thm:main}]
        If there is an integral Hodge isometry between OG6 manifolds $X$ and $Y$, then the classical bimeromorphic Torelli of OG6 \cite[Theorem~1.1]{MR} provides a bimeromorphic map $f:X\dashrightarrow Y$. The induced map
		\[
		    f^*:H^2(Y,\Z)\xrightarrow{\sim} H^2(X,\Z)
		\]
		is an integral Hodge isometry and a parallel transport operator \cite[Theorem~3.1]{Markman}.
		Proposition~\ref{prop:mr-walls}, applied also to $Y$, implies that $f^*\Kah_Y$ is a chamber of \ref{eq:kahlertype-chambers}.
		
		Both $f^*\Kah_Y$ and $\Kah_X$ lie in $\BK_X\subset\FE_X$. By standard facts about locally finite hyperplane arrangements in an open convex cone, one can choose a finite sequence of chambers
		\[
		C_0=f^*\Kah_Y,\quad C_1,\ldots,C_m=\Kah_X,
		\]
		whose successive wall crossings lie in $\FE_X$. Write $\delta_i^\perp$ for the common wall of $C_{i-1}$ and $C_i$.
		
		By Lemma~\ref{lem:reflection}, each $s_{\delta_i}$ belongs to $\Mon^2_{\mathrm{Bir}}(X)$ and maps $C_{i-1}$ to $C_i$. Set $w_0=1$ and $w_i=s_{\delta_i}\circ w_{i-1}$.
		Inductively,
		\[
		w_i(C_0)=s_{\delta_i}\bigl(w_{i-1}(C_0)\bigr) =s_{\delta_i}(C_{i-1})=C_i.
		\]
		Consequently,
		\[
		w=s_{\delta_m}\circ\cdots\circ s_{\delta_1} \in\Mon^2_{\mathrm{Bir}}(X),\qquad w(f^*\Kah_Y)=\Kah_X.
		\]
		
		Parallel transport operators are closed under composition and inversion. Hence 
		\[
		    w\circ f^*: H^2(Y,\Z) \xrightarrow{\sim} H^2(X,\Z) 
		\]
		is a Hodge parallel transport isometry sending a K\"ahler class of $Y$ to a K\"ahler class of $X$. Apply Theorem~\ref{torelli}, it gives a biholomorphism $h:X\xrightarrow{\sim} Y$. 
	\end{proof}

 As an immediate consequence, we get 
\begin{corollary}
Let $X$ and $Y$ be two OG6 manifolds. Then $X$ and $Y$ are birational if and only if $X\cong Y$. 
	\end{corollary}
\printbibliography
\end{document}